\documentclass{article}
\pdfoutput=1
\usepackage{amsfonts,amsmath,amsthm}
\usepackage{graphicx}

\newtheorem{theorem}{Theorem}
\newtheorem{proposition}[theorem]{Proposition}%
\theoremstyle{definition}
\newtheorem{definition}{Definition}%
\newtheorem{corollary}{Corollary}%

\newcommand{\pic}[4]{\vspace{1ex}\setlength{\unitlength}{.05\textwidth}
\begin{picture}(0,0)(0,0)
\put(#2){\includegraphics[#4]{#1.jpg}} 
\end{picture}\vspace*{#3cm}\vspace{1ex}} 
\def\A{\mathcal{A}}
\def\B{\mathbb{B}}
\def\BB{\mathcal{B}}
\def\E{\mathcal{E}}
\def\Har{\operatorname{\mathcal{H}}}
\def\HH{\mathcal{H}}
\def\H{\mathbb{H}}

\def\M{\mathcal{M}}
\def\Mbar{\overline{\M}}
\def\N{\mathcal{N}}
\def\P{\mathcal{P}}
\def\Q{\mathcal{Q}}
\def\R{\mathbb{R}}
\def\Sc{\operatorname{Sc}}
\def\Vec{\operatorname{Vec}}
\def\Dbar{\overline{\partial}}

\def\s{\raisebox{-2.2ex}{\rule{0pt}{4.5ex}}}

\def\sige{{\hspace{.16ex}\bf e}}
\def\sigi{{\hspace{.16ex}\bf i}}

\newcommand{\U}[2]{U_{#1}^{#2}}
\newcommand{\X}[2]{X_{#1}^{#2}}
\newcommand{\Xbar}[2]{\overline{X}_{#1}^{#2}}
\newcommand{\Xhat}[2]{\widehat{X}_{#1}^{#2}}
\newcommand{\Y}[2]{Y_{#1}^{#2}}
\newcommand{\Yt}[2]{\widetilde{Y}_{#1}^{#2}}
\newcommand{\ZZ}[2]{Z_{#1}^{#2}} 
\newcommand{\ZZhat}[2]{\widehat{Z}_{#1}^{#2}} 

\begin{document}

\begin{center} {\Large
  Bergman kernels for monogenic and contragenic
    functions\\ in the interior and exterior of a sphere}
  \bigskip

\parbox{.8\textwidth}{ 
  \begin{center}
    R. Garc\'ia-Ancona
  \end{center} \vspace{-2ex}
  {\small  \texttt{raga01@ciencias.unam.mx}
  (Facultad de Ciencias, UNAM, Mexico)},
\\
\begin{center} 
  J. Morais
\end{center}\vspace{-2ex}
{\small \texttt{joao.morais@itam.mx}
(Department of Mathematics, ITAM, Mexico)}
\\
\begin{center}
  R. Michael Porter
\end{center}\vspace{-2ex}
{\small \texttt{mike@math.cinvestav.edu.mx}
(Department of Mathematics, CINVESTAV-Quer\'etaro, Mexico)}
}
 
  \bigskip
\parbox{.85\textwidth}{Abstract. Contragenic functions are defined to be reduced-quaternion-valued
  harmonic functions which are orthogonal to all monogenic and
  antimonogenic functions in the $L^2$ norm of a given domain. The
  parallelism between the spaces of contragenic functions in the
  interior and exterior of the unit sphere in $\R^3$ is described in
  detail. Bergman reproducing kernels for the spaces of contragenic
  functions are given, mirroring the corresponding kernels for the
  spaces of vector parts of monogenic functions. Numerical examples
  are given showing the accuracy of truncations of the integral
  kernels.  A striking duality is observed between the basic interior
  contragenic functions and the vector parts of exterior monogenic
  functions, and vice versa.}
\end{center}


\medskip
 


\begin{quote} \large
\textit{In memory of our dear friend Frank Sommen, who provided us with inspiration and encouragement.}
\end{quote}


\section{Introduction}
 
This paper aims to extend and unify some fundamental relations among
harmonic, monogenic, and contragenic functions, highlighting the
analogies and differences among $L^2$ spaces of such functions in the
interior and exterior of the unit sphere in $\R^3$.

The solid spherical harmonics are given in spherical coordinates by \cite{Hobson}
\begin{align} \label{eq:defhar}
\U{n,m}{\pm}(x) = \rho^n  P_n^m(\cos \theta) \Phi^\pm_m(\phi),
\end{align}
with the terminology specified in Section \ref{sec:harmonic} below.
For  $n\ge0$, $\U{n,m}{\pm}$ is square integrable in the unit
ball $\B^\sigi$, while for $n\le-2$, it is square integrable in the
exterior  $\B^\sige$ of the ball.

Let $\A\cong\R^3$ denote the real subspace of the quaternions
$\H=\{a_0+a_1e_1+a_2e_2+a_3e_3\}$ for which $a_3=0$ (see Section \ref{sec:monogenic}). We are interested
in the space of $\A$-valued harmonic functions, that is, triples of
$\R$-valued harmonic functions in given domain. As a particular
example, a smooth $\A$-valued function $f$ is called monogenic when
$\Dbar f=0$. (This is equivalent to $f\Dbar=0$, a property which does not generally hold for $\H$-valued functions.) According to \cite{Cacao,Cacao2010}, the basic spherical
monogenic functions of degree $n$ are
\[ \X{n,m}{\pm} = \partial \U{n+1,m}{\pm} \] ($\Dbar,\partial $
defined in Section \ref{sec:monogenic}). An $\A$-valued harmonic
function, which is orthogonal to all monogenic functions and their
conjugates, is called contragenic. In \cite{Alvarez2014}, an orthogonal
basis for the subspace of contragenic functions was constructed for $\B^\sigi$, which when combined with a suitable basis for the sums of monogenic functions, gives an orthogonal basis for the space of all $\A$-valued harmonic functions. A similar construction was given in \cite{GMP} for contragenic functions in the interior of a spheroid of arbitrary eccentricity. Various conversion formulas that relate systems of harmonic and contragenic functions associated with spheroids of differing eccentricity were presented in \cite{GMP2020}, showing that there are common contragenic functions to all spheroids of all eccentricities.

The paper is organized as follows. We summarize the basic facts about harmonic functions in Section
\ref{sec:harmonic}.  In Section \ref{sec:monogenic}, we carry out the
construction of contragenic functions for the exterior domain
$\B^\sige$, highlighting the similarities and differences between
$\B^\sigi$ and $\B^\sige$. Surprisingly, there exist exterior contragenic functions whose scalar parts do not vanish, which was known not to occur in the interior of the ball. Moreover, we point out a previously unnoticed duality relation between the bases of monogenic and contragenic functions.  In
Section \ref{sec:bergman}, we give Bergman-type kernel operators for
the spaces of vector parts of monogenic functions and of contragenic
functions in $\B^\sige$, comparing them to the integral kernel for
$\B^\sigi$ given in \cite{Alvarez2014}. Altogether, we have complementary
projectors from $L^2(\B^\sigma)$ onto the corresponding subspaces.
Numerical examples are shown
in Section \ref{sec:examples}.

It is well known that the Kelvin transform $x\mapsto (1/\vert x\vert)u(x/\vert x\vert^2)$  is a bijective map which sends basic spherical harmonics in $\B^\sigi$ to those in $\B^\sige$ and vice versa \cite[Ch.\ 4]{Axler}. In \cite{Sudbery}, Sudbery proposed an inversion transform for monogenic functions $f \colon \H \to \H$ using $\overline{x}/\vert x\vert^2$ instead of $x/ \vert x\vert^2$. In \cite[p.\ 192]{Gurlebeck2},  Sudbery's transformation was generalized for monogenic, paravector-valued, and homogeneous functions in $\R^{d+1}$, $d \geq 1$. If these operations, besides monogenicity and orthogonality in $L^2$, also respected the canonical structure of space $\A$ of codimension one, most of what we have done here would be trivial. However,  Sudbery's inversion is not applicable to $\A$-valued functions. We have not yet been able to construct a useful transformation for $\A$-valued functions.


\section{Harmonic functions\label{sec:harmonic}}

We will write $\B^\sigi=\{x\in\R^3\colon\ \vert x\vert<1\}$ for the interior of
the unit ball in Euclidean space, and
$\B^\sige=\{x\in\R^3\colon\ \vert x\vert>1\}$ for the exterior domain. We write
$\Har^\sigi$, $\Har^\sige$ for the subspaces of real-valued harmonic
functions in $L^2(\B^\sigi)$, $L^2(\B^\sige)$, and with the symbol $\sigma$
to denote $\sigi$ or $\sige$, we write $\Har^\sigma(n)$ for the
collection elements of $\Har^\sigma$ which are homogeneous of degree $n$ (including $0$ so that these are 
vector subspaces).  The
basic harmonics are expressed by \eqref{eq:defhar} where
$x=(x_0,x_1,x_2)=
(\rho\cos\theta,\rho\sin\theta\cos\phi,\rho\sin\theta\sin\phi)$ with
$0 < \rho < \infty$, $0 < \theta \leq \pi$, and $0 < \phi \leq 2\pi$;
$P_n^m$ are the associated Legendre functions of integer order $m$ and
degree $n$; and $\Phi_m^+(\phi)=\cos (m\phi)$,
$\Phi_m^-(\phi)=\sin (m\phi)$.  The notation $\U{n,0}{-}$ for all $n$
is excluded because it vanishes identically; likewise $\U{n,m}{\pm}$
vanishes when $m>n$. Positive and negative degrees are related \cite[p.\ 192]{Hobson} through
the well known identity $P_{-n}^m(t)=P_{n-1}^m(t)$, giving
\begin{align} \label{eq:Uduality}
\U{n,m}{\pm}(x)= \vert x\vert^{2n+1} \U{-n-1,m}{\pm}(x) .
\end{align}
The harmonic function $ \U{-1,0}{+}(x)=1/\vert x\vert$ is excluded because it
is not in $L^2(\B^\sige)$. However, its partial derivatives are square
integrable and will be used later. Further, it should be noted
that for negative degree $n$,  negative values of the index $m$ are discarded for two separate reasons. The relation
\begin{align} \label{eq:Pminusm}
  P_n^{-m}(t) = \frac{(n-m)!}{(n+m)!} \, P_n^m(t)
\end{align}
implies that for $n\le0$, when $0\le m\le-n$ the functions $\U{n,-m}{\pm}$ duplicate the
functions $\U{n,m}{\pm}$   and therefore are not
considered. It is somewhat less known that for $m>-n$, the relation
\eqref{eq:Pminusm} is not valid, because $(n-m)!$ is infinite (properly
interpreted as a $\Gamma$ function), while $P_n^m(t)$ vanishes.  The
product in \eqref{eq:Pminusm} gives a nontrivial function $P_n^{-m}(t)$ which is interesting in itself, but the
corresponding harmonic function has a singularity when $\theta=\pi$,
so likewise we do not consider $\U{n,m}{\pm}$ for $m>-n$. In summary,
the $L^2$ spaces of harmonic functions correspond to the index sets
\begin{align*}
  J^\sigi=\{n\colon\ n\ge0\}, \ J^\sige=\{n\colon\ n\le -2\}
\end{align*}
in $\B^i$, $\B^e$ respectively.
Then for $n\in J^\sigma$, one requires $m\in I_\HH^\sigma$, where
\begin{align*}
  I_\HH^\sigi(n)=\{0\leq m\leq n\},\ 
  I_\HH^\sige(n)=\{0\leq m\leq -n-1\}.
\end{align*}
The redundancy in that the symbol $\sigma$ in $I_\HH^\sigma(n)$ is
determined by the sign of $n$ will assist in readability.
 
It is well known, and easy to verify, that for $n\in J^\sigma$, the set
\begin{align}
\{\U{n,m}{\pm} \vert_{\B^{\sigma}}\colon\ m\in I_{\HH^\sigma(n)}\} 
\end{align}
is an orthogonal basis for $\Har^\sigma(n)$, containing $2n+1$ ($\sigma=\sigi$) and $-2n-1$ ($\sigma=\sige$) elements, respectively (see \cite{Axler,Hobson}).
   The union of these graded bases over $n\in J^\sigma$ is an orthogonal basis for $\Har^\sigma$.  Therefore the coefficients of any element of $\Har^\sigma$ with respect to this basis may be calculated by use of the following result.
 
\begin{proposition} \label{prop:normU}
Let $\sigma=\sigi$ or $\sige$, and $n\in J^\sigma$,
     $m\in I_\HH^\sigma(n)$. The $L^2$ norms (with respect to the Lebesgue measure) of the solid spherical
  harmonics in $\B^\sigma$ are given by
\begin{align*}
    \|\U{n,m}{\pm}\|_\sigma^2 =
\dfrac{2\pi(1+\delta_{0,m})(\vert n+1/2\vert+m-1/2)!}{(2n+3)(2n+1)(\vert n+1/2\vert-m-1/2)!},
\end{align*}
where $\delta_{n,m}$ is the Kronecker symbol.
\end{proposition}
\begin{proof}
  The formula for $\B^\sigi$ appears in \cite{Muller,Sansone}.  For $\B^\sige$, we have a similar calculation
using \eqref{eq:defhar}.
\end{proof}


\section{Monogenic and related functions\label{sec:monogenic}}

\subsection{$\A$-valued monogenic functions}

Although $\A$ is
not closed under quaternionic multiplication (which is defined, as is
usual, so that $e_i^2=-1$ and $e_ie_j=-e_je_i$ for $i\not=j$), a great
deal of the analysis analogous to that of complex numbers
\cite{Gurlebeck2,GuerlebeckHabethaSproessig2016,MoraisHabilitation2021} holds with respect to the
Cauchy-Riemann (or Fueter) operators
\begin{align}  \label{eq:defD}
  \Dbar = \frac{\partial}{\partial x_0} +
       \sum_{i=1}^2 e_i \frac{\partial}{\partial x_i}, \quad
   \partial = \frac{\partial}{\partial x_0} -
       \sum_{i=1}^2e_i \frac{\partial}{\partial x_i}.
\end{align}
A smooth function $f$ defined in an open set of $\R^3$ is called
\textit{monogenic} when $\Dbar f=0$ identically; the terms
``hyperholomorphic'' and ``regular'' are also commonly used. We write $\M^\sigma$ for
the $\R$-subspace of $\A$-valued monogenic functions in $L^2(\B^\sigma)$,
and $\M^\sigma(n)$ for the subspace of $\M^\sigma$ homogeneous functions of degree
$n$.  In the following, the combination of $m=0$ with a superscript
``$-$'' is always excluded. For $m\in I_\HH^\sigma(n+1)$, we define
\begin{align} \label{eq:defX}
  \X{n,m}{\pm} = \partial \U{n+1,m}{\pm} .
\end{align}
(In some presentations, the operator $(1/2)\partial$ is used
instead of $\partial$ to provide a closer relation to geometric
properties of the derivative.)
These monogenic functions have not been studied in detail for negative $n$.

Let us write
\begin{align}
  I_\M^\sigma(n) =  I_\HH^\sigma(n+1) 
\end{align}
which has meaning for $n=-2$ if we allow $I_\HH^\sige(-1)=1$.

\begin{proposition} \label{prop:basisX}
  Let $n\in J^\sigma$. The collection of  functions
  \begin{align*}
    \{\X{n,m}\pm \vert_{\B^\sigma}\colon\ m\in I_\M^\sigma(n) \}
  \end{align*}
  is an
  orthogonal basis of $\M^\sigma(n)$ with respect to the scalar valued
  inner product
\begin{align} \label{eq:scalarproduct}
  \langle f,g \rangle_\sigma = \int_{\B^\sigma}
  (f_0\,g_0 + f_1\,g_1 + f_2\,g_2 ) \,dV.
\end{align}
Written out in coordinates,
\begin{align}  
\X{n,m}{\pm}&=(n+m+1)\U{n,m}{\pm} + \dfrac{1}{2}\big(\U{n,m+1}{\pm}-\alpha_{n,m}\U{n,m-1}{\pm}\big)e_1\nonumber\\
&\quad \mp\dfrac{1}{2}\big(\U{n,m+1}{\mp}+\alpha_{n,m}\U{n,m-1}{\mp}\big)e_2, \label{eq:solidspmonogenic}
\end{align}
where $\alpha_{n,m}=(n+m)(n+m+1)$.
\end{proposition}
 
The expression \eqref{eq:solidspmonogenic} takes a simpler form for
the extreme values of the indices, that is, $m \in \{0, n, n+1\}$, which is
obtained in each case by taking into account the relations
$\U{n,-1}{\pm} =\mp(1/n(n+1))\U{n,1}{\pm}$ and $\U{n,n+1}{\pm}=0$, respectively.

\begin{proof}
  The basis for $\B^i(n)$ (i.e., $n\ge0$) was given in \cite{Cacao} and \cite{Cacao2010}
  (even more, a basis is given there for full-quaternion valued
  monogenics), while the representations \eqref{eq:solidspmonogenic}
  of the coordinate functions were given in \cite{MoraisGur} for
  $\B^i(n)$. The basis for $\B^e(n)$ was given in \cite{MoraisNguyenKou2016} and used in
  \cite{MoraisHabilitation2021}.
In particular, the expression \eqref{eq:solidspmonogenic} involves the relation 
$(\partial/\partial x_0) \U{n,m}{\pm}=(n+m)\U{n-1,m}{\pm}$ which is also valid for $n\leq-2$.
\end{proof}

It is worth clarifying that the references cited in this proof justify
that the $\X{n,m}{\pm}$ form a basis via the observation that the
operator $\partial$ in \eqref{eq:defX} is an isomorphism from
$\Har^\sigma(n+1)$ onto $\M^\sigma(n)$.  From
\eqref{eq:solidspmonogenic} comes the well known fact that the set of
scalar parts of monogenic functions in the interior of the ball,
$\{\Sc f\colon\ f\in\M^\sigi\}$ is equal to $\Har^\sigi$. Perhaps
surprisingly, it also shows that $\{\Sc f\colon\ f\in\M^\sige\}$ is
equal to the orthogonal complement in $\Har^\sige$ of the subspace
$\{\U{n,-n+1}{\pm}\colon\ n\in J^\sige\}$ .

The following result, verified by a routine calculation, highlights again the computational similarity between the interior and exterior cases.

\begin{proposition} \label{prop:normX}
  Let  $n\in J^\sigma$,  $m\in I_\M^\sigma(n)$. Then
\begin{align*}
  \|\X{n,m}{\pm}\|_\sigma^2 = 
  \dfrac{2\pi(n+1) (\delta_{0,m}+1) (\vert n+3/2\vert+m-1/2)!}{(2n+3)(\vert n+3/2\vert-m-1/2)!}. 
\end{align*}
\end{proposition}

The basic monogenic functions satisfy the following ``Appell-type property'' with respect to the differential operator $\partial$ \cite{CGB2006}:
\begin{align} \label{eq:appell}
  \partial \X{n,m}{\pm} = 2(n+m+1) \X{n-1,m}{\pm}.
\end{align}
This is well known for $n\ge0$,
with the understanding that it is
for $0\le m\le n$ because $\partial \X{n,n+1}{\pm}=0$. It is easily seen that \eqref{eq:appell} also holds for $n\le-3$. Thus
$\partial\colon\M^\sigma(n)\to\M^\sigma(n-1)$ is surjective for
$\sigma=\sigi$ (where the dimension of the subspace of monogenic polynomials increases with the degree $n$) and injective for $\sigma=\sige$ (where the dimension increases with $-n$).


\subsection{Ambigenic functions}

It is easily seen that a function $f$ will be \textit{antimonogenic}; that is,
$\partial f=0$, if and only if the conjugate function $\overline{f}$
is monogenic.  Functions which are simultaneously monogenic and
antimonogenic, i.e.,\ in the space $\M^\sigma\cap\Mbar^\sigma$, are
called \textit{monogenic constants} \cite{CGB2006}. Monogenic
constants do not depend on $x_0$ and can be expressed as
$f = a_0+f_1e_1+f_2e_2$ where $a_0\in\R$ is a constant, and $f_1-if_2$
is an ordinary holomorphic function of the complex variable
$x_1+ix_2$.  

In $\B^\sigi$, the monogenic constants of degree $n\ge1$
are generated by, $\X{n,n+1}{\pm}$; i.e.,
$\M^\sigma(n)\cap\Mbar^\sigma(n)$ has dimension 2. For degree $n=0$,
the dimension is 3 because all constant functions (in particular
$1$, $e_1$, $e_2$) are both monogenic and antimonogenic.  For
negative degree the situation is somewhat different:

\begin{proposition} \label{prop:noexteriormc}
  There are no nonzero square-integrable monogenic constants in $\B^\sige$.
\end{proposition}
\begin{proof}
  The argument holds in fact for all domains exterior to a bounded
  set.  Suppose that $f\in\M^\sige\cap\Mbar^\sige$ is not identically
  zero.  Since $f(x )$ does not depend on the variable $x_0$, we may
  take a point $p=(p_0,p_1, p_2)\in\B^\sige$ such that
  $\vert f(x)\vert>\epsilon$ in a neighborhood of $p$. Then $\vert f(x)\vert>\epsilon$
  in the half-infinite cylinder formed of points $(x_0,x _1,x_2)$ such that
  $(x_1,x_2)$ varies in a neighborhood of $(p_1,p_2)$ while
  $x_0\ge p_0$ in case $p_0>0$, or $x_0\le p_0$ in case $p_0<0$. In
  either case, we have $f\not\in L^2(\B^\sige)$, contrary to
  hypothesis.
\end{proof}

Functions in $\M^\sigma+\Mbar^\sigma$, that is, those which can be
expressed as the sum of a monogenic and an antimonogenic function, are
termed \textit{ambigenic}. All monogenic and all antimonogenic
functions are trivially ambigenic. The relations
$\Sc f=(1/2)(f+\overline{f})$ and $\Vec f = (1/2)(f-\overline{f})$
show that the scalar parts and the vector parts of monogenic functions
are ambigenic.

Define the ambigenic functions
\begin{align}
  \Y{n,m}{\sigma, \pm} &= \X{n,m}{\pm},   \nonumber\\
  \Yt{n,m}{\sigma,\pm} &= \Xbar{n,m}{\pm} -
  \beta_{n,m} \X{n,m}{\pm}
\end{align}
in $\B^\sigma$, where 
\[ \beta_{n,m} = \beta^\sigma_{n,m} =
  \dfrac{\langle\X{n,m}{+},\Xbar{n,m}{+}\rangle_\sigma}{\|\X{n,m}{+}\|^2_\sigma} =
  \dfrac{\langle\X{n,m}{-},\Xbar{n,m}{-}\rangle_\sigma}{\|\X{n,m}{+}\|^2_\sigma} = \dfrac{n-2m^2+1}{(n+1)(2n+1)}.
\]

\begin{proposition}  \label{prop:ambibasis}
For each $n\in J^\sigma$, the collection of   functions
 \begin{align*}
   \{ \Y{n,m}{\sigma,\pm},\  \Yt{n,m}{\sigma,\pm}\colon\ m\in I^\sigma(n+1) \}
 \end{align*}
 is an orthogonal basis for the subspace of ambigenic functions $\M^\sigma(n)+\Mbar^\sigma(n)$. The union of these bases over $n\in J^\sigma$ is an orthogonal
 basis of $\M^\sigma+\Mbar^\sigma$.
\end{proposition}
 
\begin{proof} The proof for $\B^\sigi$ was given in \cite{Alvarez2014}, so
  we consider $\B^\sige$.  The orthogonality is again verified by
  a tedious calculation. Since the space of monogenic constants
  is trivial, the dimensions of the spaces of homogeneous exterior
  ambigenics are given by
\begin{align*}
    \dim(\M^\sige(n)+\Mbar^\sige(n))=\dim\M^\sige(n)+\dim\Mbar^\sige(n).
\end{align*}
An unavoidable complication is that not quite all antimonogenics are
orthogonal to the monogenics. More precisely, the inner products on
$\B^\sige$ of basic monogenics with basic antimonogenics are given by
\begin{align*}
    \langle \Xbar{n_1,m_1}{+}, \X{n_2,m_2}{-}\rangle_\sige &= \langle \Xbar{n_1,m_1}{-},\X{n_2,m_2}{+}\rangle_\sige = 0 \;\,  \mbox{ when }  n_1 \neq n_2
 \mbox{ or }  m_1 \neq m_2 ,
\end{align*}
while 
\begin{align*}
    \langle \Xbar{n_1,m_1}{+}, \X{n_2,m_2}{+}\rangle_\sige &= \langle \Xbar{n_1,m_1}{-}, \X{n_2,m_2}{-}\rangle_\sige \\
    &=  (\|\Sc\X{n_1,m_1}{\pm}\|^2_\sige -
    \|\Vec\X{n_1,m_1}{\pm}\|^2_\sige )\,\delta_{n_1,n_2}\delta_{m_1,m_2},
\end{align*}
 where we write
\begin{align}
    \Vec f= f - \Sc f = f_1e_1+f_2e_2.
\end{align}
  
Straightforward calculations based on Proposition \ref{prop:normU}
  show that in particular
\begin{align*}
    \langle \Xbar{n,m}{+},\X{n,m}{+}\rangle_\sige
    =  \dfrac{2\pi(n-2m^2+1)(1+\delta_{0,m})(-(n-m+2))!}{(2n+3)(2n+1)(-(n+m+2))!} .
\end{align*}
The orthogonality claimed for the $\Y{n,m}{\sigma,\pm}$,
  $\Yt{n,m}{\sigma,\pm}$ follows readily from this..
\end{proof}
        
We will not use the basis
$\{\Y{n,m}{\sigma,\pm},\ \Yt{n,m}{\sigma,\pm}\}$ in the following. It
has been included to enable explicit computational decompositions in
$(\Har^\sigma)^3$.  For use in the next
section, we note that direct calculations also show that
\begin{align} \label{eq:normvecX}
 \|\Vec\X{n,m}{\pm}\|_\sigma^2 = 
  \dfrac{2\pi(n^2+m^2+n) (\delta_{0,m}+1) (\vert n+3/2\vert+m-1/2)!}{(2n+3)(2n+1)(\vert n+3/2\vert-m-1/2)!}.
\end{align}

\subsection{Contragenic functions\label{sec:contragenic}}

It was observed in \cite{Alvarez2014} that, contrary to a basic fact of
complex numbers or indeed of any full Clifford algebra, not every harmonic
function is ambigenic in the sense of $\A$-valued functions that we
have given here. This  is related to the fact that $\A$ is not
closed under quaternionic multiplication.
Since the three real component functions of an ambigenic
function are harmonic, we can regard it as an element of $(\Har(n))^3$.
One says that $f$ is \textit{contragenic} when it is orthogonal to all
ambigenic functions in the domain of definition under consideration.
We have the spaces of homogeneous functions of degree $n$,
\begin{align}
  \N^\sigma(n)= (\M^\sigma(n)+\Mbar^\sigma(n))^\perp\subseteq
  (\Har^\sigma(n))^3 ,
\end{align}
with respect to the inner product \eqref{eq:scalarproduct} as
always. Since constant functions are ambigenic, the only contragenic
function of degree 1 in $\B^i$ is identically zero, so we do not speak
of $\N^\sigi(0)$.

In the light of the comment following the proof of Proposition
\ref{prop:basisX}, $\Sc f$ ranges over all of $\Har^\sigi$ as $f$
ranges over $\M^\sigi$, which in the light of
\eqref{eq:scalarproduct} implies that \textit{the scalar part of every
  contragenic function in $\B^i$ is zero}.  In general, a given
harmonic function $g_1e_1+g_2e_2$ with vanishing scalar part is
contragenic if and only if it is orthogonal in
$L^2(\{0\}\oplus(\Har^\sigi)^2)$ to the subspace
\begin{align*}
  \Vec\M^\sigi = \{\Vec f\colon\ f\in \M^\sigi \}=\Vec\Mbar^\sigi .
\end{align*} 

In $\B^\sige$, the situation is somewhat different. 

\begin{proposition} \label{prop:scalarcontragenics}
  Let $f$ be a contragenic function in $\B^\sige$. Then $\Sc f$
  lies in the subspace of $\Har^\sige$ generated by the
  particular exterior harmonics $\U{n,-n-1}{\pm}$, $n\in I_\HH^\sige$.
\end{proposition}

\begin{proof}
  Write $\Sc f = \sum a_{n,m}^\pm \U{n,m}{\pm}$ with indices ranging
  over $n\in J^\sige$, $m\in I_\HH^\sige(n)$. For every term in
  this sum for which $m\le -n-2$, take a monogenic function
  $f_{n,m}^\pm$ whose scalar part is $\U{n,m}{\pm}$.  Thus
  $\U{n,m}{\pm} + 0e_1+0e_2=(1/2)(f_{n,m}^\pm+\overline{f_{n,m}^\pm})$
  is ambigenic whenever $m\not=-1$. For such values of $(n,m)$, by
  definition of contragenic,
  \begin{align*}
    0 =\langle f, f_{n,m}^\pm \rangle =
    \int_{\B^\sige} (\Sc f) \, U_{n,m}^\pm \, dV.
  \end{align*}
  Thus $\Sc f$ is orthogonal to all of the basic harmonics
  except $\U{n,-n-1}{\pm}$.
\end{proof}

 We will use the index sets  $I_\N^\sigma(n)$ given by 
\begin{align*}
  I_\N^\sigma(n) &= \{m\colon\ 0\le m\le \vert n-1\vert\} \quad (n\ge 1 \mbox{ or } n\leq-2)
\end{align*}
 
\begin{definition} \label{def:basiccontra} Let
  $n\in J^\sigma\setminus\{0\}$. The \textit{basic contragenic
    functions} are the following: for $m\in I_\N^\sigma$, the
  vector-valued contragenics
\begin{align*} 
 \ZZ{n,0}{+}&=\U{n,1}{-}e_1-\U{n,1}{+}e_2,\nonumber\\  
  \ZZ{n,m}{\pm}&=\big(\U{n,m+1}{\mp}+\alpha_{n+1,-m}\U{n,m-1}{\mp}\big)e_1\mp\big(\U{n,m+1}{\pm}-\alpha_{n+1,-m}\U{n,m-1}{\pm}\big)e_2 
\end{align*}
where $1\leq m\leq \vert n-1/2\vert-1/2$, while for $\sigma=\sige$, additionally the scalar contragenics
\begin{align*} \label{eq:basiccontragenics}
  \ZZ{n,-n+1}{\pm}&=\U{n,-n-1}{\pm}.  
\end{align*}
\end{definition}

In parallel with the other notational conventions, there is no
``$\ZZ{n,0}{-}$''.
It is perhaps surprising that the coefficient $\alpha_{n+1,-m}$
works for both cases $\sigma=\sigi$ and  $\sigma=\sige$. 

\begin{center}
\begin{table}[h!]
    \[
\begin{array}{|l|l|}    
\hline
 \s\ZZ{-2,0}{+} & -\frac{x_2}{(x_0^2+x_1^2+x_2^2)^{3/2}} e_1 + \frac{x_1}{(x_0^2+x_1^2+x_2^2)^{3/2}}e_2 \\[.5ex]\hline
 \s\ZZ{-2,1}{+} &  \frac{6 x_0}{(x_0^2+x_1^2+x_2^2)^{3/2}}e_2 \\ 
 \s\ZZ{-2,1}{-} & \frac{6 x_0}{(x_0^2+x_1^2+x_2^2)^{3/2}}e_1\\[.5ex]\hline  \s\ZZ{-2,2}{+} &   -\frac{12 x_2}{(x_0^2+x_1^2+x_2^2)^{3/2}}e_1 -\frac{12 x_1}{(x_0^2+x_1^2+x_2^2)^{3/2}} e_2\\ 
 \s\ZZ{-2,2}{-} & -\frac{12 x_1}{(x_0^2+x_1^2+x_2^2)^{3/2}}e_1 +\frac{12 x_2}{(x_0^2+x_1^2+x_2^2)^{3/2}}e_2 \\[.5ex]\hline \s\ZZ{-2,3}{+} &  -\frac{x_1}{(x_0^2+x_1^2+x_2^2)^{3/2}} \\ 
  \s\ZZ{-2,3}{-} &  -\frac{x_2}{(x_0^2+x_1^2+x_2^2)^{3/2}}\\[.5ex]\hline\hline
  \s\ZZ{-3,0}{+} &  -\frac{3 x_0 x_2}{(x_0^2+x_1^2+x_2^2)^{5/2}}e_1+\frac{3 x_0 x_1}{(x_0^2+x_1^2+x_2^2)^{5/2}}e_2\\\hline
  \s\ZZ{-3,1}{+} &  \frac{6 x_1 x_2}{(x_0^2+x_1^2+x_2^2)^{5/2}}e_1+\big(\frac{6 (2 x_0^2-x_1^2-x_2^2)}{(x_0^2+x_1^2+x_2^2)^{5/2}}-\frac{3 (x_1-x_2) (x_1+x_2)}{(x_0^2+x_1^2+x_2^2)^{5/2}}\big)e_2 \\ 
  \s\ZZ{-3,1}{-} &  \big(\frac{3 (x_1-x_2) (x_1+x_2)}{(x_0^2+x_1^2+x_2^2)^{5/2}}+\frac{6 (2 x_0^2-x_1^2-x_2^2)}{(x_0^2+x_1^2+x_2^2)^{5/2}}\big)e_1+\frac{6 x_1 x_2}{(x_0^2+x_1^2+x_2^2)^{5/2}}e_2 \\[.5ex]\hline
 \s\ZZ{-3,2}{+} &  -\frac{60 x_0 x_2}{(x_0^2+x_1^2+x_2^2)^{5/2}}e_1-\frac{60 x_0 x_1}{(x_0^2+x_1^2+x_2^2)^{5/2}}e_2\\ 
 \ZZ{-3,2}{-} & -\frac{60 x_0 x_1}{(x_0^2+x_1^2+x_2^2)^{5/2}}e_1+\frac{60 x_0 x_2}{(x_0^2+x_1^2+x_2^2)^{5/2}}e_2\\[.5ex]\hline
  \s\ZZ{-3,3}{+} & \frac{180 x_1 x_2}{(x_0^2+x_1^2+x_2^2)^{5/2}}e_1+\frac{90 (x_1-x_2) (x_1+x_2)}{(x_0^2+x_1^2+x_2^2)^{5/2}}e_2\\ 
  \s\ZZ{-3,3}{-} &  \frac{90 (x_1-x_2) (x_1+x_2)}{(x_0^2+x_1^2+x_2^2)^{5/2}}e_1-\frac{180 x_1 x_2}{(x_0^2+x_1^2+x_2^2)^{5/2}}e_2\\[.5ex]\hline
 \s\ZZ{-3,4}{+} &  \frac{3 (x_1-x_2) (x_1+x_2)}{(x_0^2+x_1^2+x_2^2)^{5/2}} \\ 
 \s\ZZ{-3,4}{-} &  \frac{6 x_1 x_2}{(x_0^2+x_1^2+x_2^2)^{5/2}}\\[.5ex]\hline  
\end{array} \]
\caption{External spherical contragenic basis functions $\ZZ{n,m}{\pm}$ of degrees of homogeneity $n=-2,-3$.}
  \label{tab:firstcontragenics1}
\end{table}
\end{center}

\begin{theorem} \label{th:contragenicbasis}
  Let $n\in J^\sigma\setminus\{0\}$.   The basic contragenic functions
  given in Definition \ref{def:basiccontra} form an orthogonal basis for
    $\N^\sigma(n)$. 
\end{theorem}

\begin{proof}
  The statement for $\B^\sigi$ was given in \cite{Alvarez2014}. For
  $\B^\sige$, the orthogonality of the $\ZZ{n,m}{\pm}$ with respect to
  each other and with respect to the basic monogenic functions and
  their conjugates is proved by a similar calculation; the only
  new idea is that harmonics mentioned in Proposition
  \ref{prop:scalarcontragenics} are in fact contragenic, because
  \begin{align*}
    &\langle \U{n,-n-1}{\pm} , \X{n',m'}{\pm} \rangle_\sige \\
    &= (n'+m'+1) \langle \U{n,-n-1}{\pm} , \U{n',m'}{\pm} \rangle_\sige
    + \langle 0e_1, [\X{n',m'}{\pm}]_1e_1\rangle_\sige
    + \langle 0e_2, [\X{n',m'}{\pm}]_2e_2\rangle_\sige \\
    &= 0,
  \end{align*}
by orthogonality of the solid spherical harmonics.

  Then a
  dimension count shows that they form a basis. Specifically, we
  summarize the dimensions over $\R$ of the relevant spaces of
  homogeneous functions in Table \ref{tab:dimensions}. The values for
  $\B^\sigi$, i.e., for $n\ge0$, are found in \cite{Alvarez2013}.  In
  this table, the dimension of $\M^\sigma(n)+\Mbar^\sigma(n)$ is twice
  the dimension of $\M^\sigma(n)$ less the dimension of
  $\M^\sigma(n)\cap\Mbar^\sigma(n)$. The dimensions
  $\dim(\M^\sigma(n)+\Mbar^\sigma(n))$ and $\dim \M^\sigma(n)$ sum to
  $\dim (\Har^\sigma(n))^3$, which implies that
   the basic contragenics are indeed bases.
\end{proof}

 \begin{table}[th!]
\[ \begin{array}{|c||c|c||c|}
\hline
 & n=0 & n\ge 1 &  n\le-3 \\\hline
  \s \Har^\sigma(n) &1 & 2n+1 & -(2n+1) \\[.5ex] 
  \s (\Har^\sigma(n))^3 &3 & 6n+3   & -(6n+3)  \\[.5ex] 
  \s \M^\sigma(n),\ \Mbar^\sigma(n) &3 & 2n+3   &-(2n+3) \\[.5ex] 
  \s  \M^\sigma(n)\cap\Mbar^\sigma(n) &3 &  2  & 0  \\[.5ex]
  \s  \M^\sigma(n)+\Mbar^\sigma(n) &3 & 4n+4  & -(4n+6)  \\[.5ex]
  \s  \N^\sigma(n)&0  & 2n-1 &  -(2n-3) \\[.5ex] \hline 
\end{array}
\]
\caption{Dimensions of spaces of homogeneous harmonic functions,
where $\sigma=\sigi$ for $n\ge0$ while $\sigma=\sige$ for $n\le-2$.}
  \label{tab:dimensions}
\end{table}

The respective direct sums of the finite dimensional spaces
$\Har^\sigma(n)$, $\M^\sigma(n)$, $\N^\sigma(n)$ over $n\in J^\sigma$
give the full Hilbert spaces $\Har^\sigma$, $\M^\sigma$, $\N^\sigma$.
The following fact, also verified by direct computation, is necessary
in order to compute coefficients of triples of harmonic functions with
respect to the basis
$\{\Y{n,m}{\sigma,\pm},\ \Yt{n,m}{\sigma,\pm},\ \ZZ{n,m}{\pm}\}$.

\begin{corollary}
  Let $n\in J^\sigma\setminus\{0\}$ and $m\in I_\N^\sigma$. The norms of
  the vector-valued basic contragenics are
\begin{align} \label{eq:Znorms}
 \| \ZZ{n,m}{\pm}\|_\sigma^2 = \dfrac{8\pi(n^2+m^2+n)(\delta_{0,m}+1)(\vert n-1/2\vert+m-1/2)!}{(2n+3)(2n+1)(\vert n-1/2\vert-m-1/2)!},
\end{align}
while for the scalar contragenics, the norms are
\begin{align*}
\|\ZZ{n,-n+1}{\pm}\|_\sige^2 = \dfrac{2\pi \, \Gamma(-2n-1)}{(2n+3)(2n+1)}.
\end{align*}
\end{corollary}

 
\subsection{Duality relation}

The following duality relation is obtained by observation of
the explicit formulas for basic monogenics and contragenics.
We see no intrinsic reason to suppose beforehand that such a relation should hold. Recall
that $U_{n,m}^\pm$ are defined on both $\B^i$ and $\B^e$. For 
functions of the form $f=f_1e_1+f_2e_2$, we have the linear
involution
\begin{align} \label{eq:fstar}
  f^* = f_2e_1 + f_1e_2 = -e_1fe_2.
\end{align}

\begin{proposition} \label{prop:duality}
  Let $n\in J^\sigma\setminus\{0\}$, and let $m\in I_\N^\sigi(n)$
  when $n>0$; otherwise, let  $m\in I_\N^\sige(n+1)$.
  Then
\begin{align*}
\ZZ{n,0}{\pm} &= \mp \rho^{2n+1}(\X{-n-1,0}{\pm})^*, \\
\ZZ{n,m}{\pm} &= \mp2\rho^{2n+1}(\X{-n-1,m}{\pm})^*.
\end{align*}
\end{proposition}

\begin{proof}
  Observe that $\alpha_{n+1,-m}=\alpha_{-n-1,m}$. Therefore by
  \eqref{eq:Uduality}, it follows that
  \begin{align*}
    \ZZ{n,m}{\pm} &=
    \vert x\vert^{2n+1}\Bigl(\big(\U{-n-1,m+1}{\mp}+\alpha_{-n-1,m}\U{-n-1,m-1}{\mp}\big)e_1 \\
    & \qquad \qquad \quad \mp\big(\U{-n-1,m+1}{\pm}-
     \alpha_{-n-1,m}\U{-n-1,m-1}{\pm}\big)e_2\Bigr) \\
    &= \mp2 \vert x\vert^{2n+1}\Big([\X{-n-1,m}{\pm}]_2e_1+[\X{-n-1,m}{\pm}]_1e_2\Big) 
  \end{align*}
  as required.
\end{proof}

The duality relation does not apply to the scalar-valued contragenics
$\ZZ{n,-n+1}{\pm}$.

The normalized elements of $\Vec\M$ and $\Vec\N$ will be written
as
\begin{align}
  \Vec\Xhat{n,m}{\pm} = \frac{\Vec \X{n,m}{\pm}}{\|\Vec \X{n,m}{\pm}\|_\sigma}, \quad
   \ZZhat{n,m}{\pm} = \frac{\ZZ{n,m}{\pm}}{\|\ZZ{n,m}{\pm}\|_\sigma}, 
\end{align}
using \eqref{eq:normvecX} and \eqref{eq:Znorms}, where as always
$\sigma=\sigi$ or $\sige$ according to whether $n\ge0$ or
$n\le-2$. Thus the vector parts of the normalized monogenics are
$\Vec\Xhat{n,m}{\pm}=[\Xhat{n,m}{\pm}]_1e_1+[\Xhat{n,m}{\pm}]_2e_2$,
where the components $[\Xhat{n,m}{\pm}]_1$, $[\Xhat{n,m}{\pm}]_2$ can
be read off from \eqref{eq:solidspmonogenic}. Similarly, by the
duality we have the expression in components of the normalized
vectorial contragenics
$\ZZhat{n,m}{\pm}=[\ZZhat{n,m}{\pm}]_1e_1+[\ZZhat{n,m}{\pm}]_2e_2$.


\section{Bergman reproducing kernels\label{sec:bergman}}

Bergman reproducing kernels were introduced in \cite{Bergman1950,BergmanSchiffer1953} and
currently enjoy a wide-reaching theory \cite{AN1950,DMP2021,FY2013,HZ2009,KT2022,MW2014,MSKS2020,Paulsen2016}. A Bergman
kernel for $\M^\sigi$ was presented (in a different but equivalent
context) in \cite{GLS2010}. A Bergman reproducing kernel was given in
\cite{Alvarez2014} for $\Vec\M^\sigi$. In this section, we complete
this topic by giving a Bergman kernel for $\Vec\M^\sige$ as well as
for both cases of $\N^\sigma$.  Since from
now on we are interested exclusively in vector valued functions, we write
\begin{align*}
  \A_2=\R e_1+\R e_2\subseteq\A,
\end{align*}
with the natural restriction of the scalar-valued inner product.

For clarity we outline the general construction in this context.
  Given an
orthonormal basis $\{\psi_k\}_{k=1}^\infty$ of a closed subspace $E$
of the $\R$-linear space $L^2(\Omega,\A_2)$, say
$\psi_k=\psi_{1,k}e_1+ \psi_{2,k}e_2$, consider the matrix of
functions
\begin{align} \label{eq:bmatrix}
 \BB(x,y) = \begin{pmatrix}
  \sum_k\psi_{1,k}(x)\psi_{1,k}(y) & \sum_k\psi_{1,k}(x)\psi_{2,k}(y) \\[1ex]  \sum_k\psi_{2,k}(x)\psi_{1,k}(y) & \sum_k\psi_{2,k}(x)\psi_{2,k}(y) 
 \end{pmatrix} ,
\end{align}
which satisfies the symmetry relation $\BB(y,x)=\BB(x,y)^t$ with
$(\ )^t$ denoting the transpose.  The quadratic form determined by
this matrix is used to define a linear operator
$\BB\colon L^2(\Omega,\A_2) \to E$ as
follows.  For $f=f_1e_1+f_2e_2\in L^2(\Omega,\A_2)$, write
\begin{align} \label{eq:bergmangeneral}
 \BB[f](x) = \int_\Omega \begin{pmatrix} e_1 \\ e_2 \end{pmatrix}^t
 \BB(x,y)  \begin{pmatrix} f_1(y) \\ f_2(y) \end{pmatrix} d V_y.
\end{align}
  Then the operator $\BB$
satisfies the reproducing property
\begin{align}
   \BB[f](x) = f(x)
\end{align}
for every $x\in\Omega$. The verification is obtained from the basis
representation $f=\sum_{k=1}^\infty a_k\psi_k$ ($a_k\in\R$),
and its main step is the observation that
 \begin{align}
  \BB[f](x) &=  \int_\Omega\bigg(   
   \sum_k  \psi_{1,k}(x)\psi_{1,k}(y)\sum_la_l\psi_{1,l}(y) \nonumber \\
   &\qquad \qquad +
   \sum_k  \psi_{1,k}(x)\psi_{2,k}(y)\sum_la_l\psi_{2,l}(y)
    \bigg) e_1 \, dV_y  \nonumber \\
  &\ +  \int_\Omega\bigg(
   \sum_k \psi_{2,k}(x) \psi_{1,k}(y)\sum_la_l\psi_{1,l}(y) \nonumber \\
   &\qquad \qquad +
   \sum_k \psi_{2,k}(x) \psi_{2,k}(y)\sum_la_l\psi_{2,l}(y)
      \bigg) e_2 \, dV_y,    \label{eq:Bfcalc}
\end{align}
in which one factors out $\psi_{1,k}(x)$ and $\psi_{2,k}(x)$ and then applies the orthogonality relation
\begin{align*}
  \int_\Omega( \psi_{1,k}(y)\psi_{1,l}(y) + \psi_{2,k}(y)\psi_{2,l}(y) ) \, dV_y = \delta_{k,l} 
\end{align*}
after interchanging the order of integration and summation.  This
discussion depends on the convergence of the series, which will be
discussed below in particular cases. Given the convergence, it is seen
that the kernel $\BB(x,y)$ remains unchanged when one uses any other
orthonormal basis in place of $\{\psi_k\}$.


\subsection{Bergman kernel\label{sec:bergmanvecM} for $\Vec\M^\sigma$}
Elements of $\Vec\M^\sigma$ are characterized \cite{Alvarez2013} as those harmonic
functions $f=f_1e_1+f_2e_2$ of three variables satisfying the
differential equation 
\begin{align} \label{eq:VecMequation}
  \frac{\partial}{\partial x_2}f_1 =
  \frac{\partial}{\partial x_1}f_2 .
\end{align}
We take for $\{\psi_k\}$ the orthonormal basis 
\begin{align*}
  \{\Vec\Xhat{n,m}{\pm} \colon\ n\in J^\sigma,\ m\in I_\M(n)\}
\end{align*}
of $\Vec\M^\sigma(n)$ in $\B^\sigma$, with the understanding that
$\Vec\Xhat{n,0}{\sigi,+}$ is not included because they vanish.
(However, $\Vec\Xhat{0,1}{\pm}$ must be included.) Similarly,
$\Vec\Xhat{-2,0}{-} = 0$. The associated Bergman operator
\eqref{eq:bergmangeneral} can be described as follows.  Define the
functions $b_{\M,1}^\sigma$, $b_{\M,2}^\sigma$ for
$x,y\in\B^\sigma$, which take values in $\A_2$:
\begin{align}
  b_{\M,j}^\sigma(x,y) &= \sum_{n\in J^\sigma}
   \sum_{m\in I_\M^\sigma(n)} \big(
          [\Xhat{n,m}{+,\sigma}]_j(x)\,\Vec\Xhat{n,m}{+}(y) \nonumber \\[-1.5ex]
   &\qquad \qquad \qquad \qquad + 
          [\Xhat{n,m}{-,\sigma}]_j(x)\,\Vec\Xhat{n,m}{-}(y)\big)
   \nonumber \\[2ex]
   &= \sum_{n\in J^\sigma}
   \sum_{m\in I_\M^\sigma(n)} \big(
          [\Xhat{n,m}{+,\sigma}]_j(y)\,\Vec\Xhat{n,m}{+}(x) \nonumber \\[-1.5ex]
   &\qquad \qquad \qquad \qquad + 
          [\Xhat{n,m}{-,\sigma}]_j(y)\,\Vec\Xhat{n,m}{-}(x)\big)
      \label{eq:bergmanopmonog}
\end{align}
($j=1,2$).  The convergence of the series can be verified using
Proposition \ref{prop:normX}, but it follows more easily by general
results used in the theory of Bergman kernels. This approach uses
Parseval's formula applied to an arbitrary $L^2$ orthogonal basis,
combined with the well-known boundedness of the evaluation functional
\begin{align*}
  \vert f(x)\vert \le C_\delta\|f\|_2,
\end{align*}
where $C_\delta$ is a constant determined by the distance $\delta$
of $x$ to the boundary of the domain.
The Bergman operator $\BB_\M^\sigma$ for $\Vec\M^\sigma$ can be
expressed as
\begin{align} \label{eq:bergmanop}
    \BB_\M^\sigma[f](x) &=
    \left(  \int_{\B^\sigma} \Sc(b_{\M,1}^\sigma(x,y)f(y)) \,dV_y\right)e_1 \nonumber \\
    &\quad +
    \left( \int_{\B^\sigma} \Sc(b_{\M,2}^\sigma(x,y)f(y)) \,dV_y\right)e_2.
\end{align}
Note that each integrand contains a multiplication involving $e_1$ and
$e_2$ before taking the scalar part.  From the definition it is
immediate that $b_{\M,j}^\sigma(x,y)$ are elements of $\Vec\M^\sigma$
as functions of $x$, and then by applying the criterion
\eqref{eq:VecMequation} to the representation \eqref{eq:Bfcalc} via
differentiation across the integral sign, it is seen that
$\BB_\M^\sigma[f]\in\Vec\M^\sigma$.

 \begin{proposition} $\BB_\M^\sigma$ projects the Hilbert space
   $L^2(\B^\sigma,\A_2)$ orthogonally onto the subspace
   $\Vec\M^\sigma$. Thus
  \begin{align*}
    \BB_\M^\sigma[f]=f
  \end{align*}
  if and only if $f\in\Vec\M^\sigma$. If $f$ is harmonic, then $\BB_\M^\sigma[f]=0$ if and only if
  $f\in\N^\sigma$.
\end{proposition}

\begin{proof}
  The proof for $\B^\sigi$ follows by standard Hilbert space arguments
  based on the orthogonality of $\Vec\X{n,m}{\pm} \vert_{\B^\sige}$, and was
  given in \cite{Alvarez2014}. For $\B^\sige$ it follows by the same
  reasoning.
\end{proof}
  
By Proposition \ref{prop:duality} and the homogeneity of the basic
functions, the Bergman kernel for $\B^\sigma$ can be expressed in
terms of the basic contragenic functions. Note that because of \eqref{eq:fstar}, the 
indices $1$ and $2$ will be interchanged as indicated by the notation below:
\begin{corollary} Write $j'=3-j$, $j=1,2$. Then the terms in the Bergman
  reproducing kernel for $\Vec\M^\sigma$ are given by

  \begin{align*}
    b_j^\sigma(x,y)  = -\dfrac{1}{2}\sum_{n\in J^\sigma}
    \vert x\vert^{2n+1}\vert y\vert^{2n+1}  \sum_{m\in I_\N^\sigma(n)} 
  & \bigg(\bigl[\ZZhat{-n-1,m}{+,\sigma}\bigr]_{j'}(x)\,
     \ZZhat{-n-1,m}{+,\sigma}(y) \\
  & + \bigl[\ZZhat{-n-1,m}{-,\sigma}\bigr]_{j'}(x)\,
     \ZZhat{-n-1,m}{-,\sigma}(y)\bigg),
  \end{align*}  
\end{corollary}

Observe that $\Vec\ZZhat{-n-1,-n}{+,e}=0e_1+0e_2$ in this sum. We close this section with an observation on the orthogonal
decomposition
\begin{align*}
  L^2(\B^\sigma,\A_2) = \Vec\M^\sigma\oplus \N^\sigma \oplus
    \Har(\B^\sigma,\A_2)^\perp.
\end{align*}
The operators
\begin{align*}
  \P^\sigma = \BB_\M^\sigma + \BB_\N^\sigma,\quad
  \Q^\sigma = I - \P^\sigma
\end{align*}
are complementary orthogonal projectors on $L^2(\B^\sigma,\A_2)$,
\begin{align*}
  \P^\sigma&\colon L^2(\B^\sigma,\A_2)\to\Har(\B^\sigma,\A^2), \\
  \Q^\sigma&\colon L^2(\B^\sigma,\A_2)\to \Har(\B^\sigma,\A^2)^\perp.
\end{align*}
The relations $(\P^\sigma)^2=(\Q^\sigma)^2=I$,
$\P^\sigma\Q^\sigma=\Q^\sigma\P^\sigma=0$ are easily verified from the
corresponding properties for $ \BB_\M^\sigma$ and $\BB_\N^\sigma$.


\subsection{Bergman kernel\label{sec:bergmancont} for $\N^\sigma$}

Here we give the analogous construction for contragenic
functions, based on the orthogonal basis $\{Z_{n,m}^\pm\}$. Being an orthogonal complement, $\N^\sigma$ is a closed subspace of
$L^2(\B^\sigma, \A_2)$. We define
\begin{align}\label{eq:bergmankercontrag}
  b_{\N,j}^\sigma(x,y) &= (-)^\sigma\sum_{n\in J^\sigma}
   \sum_{m\in I_\N^\sigma(n)} \big(
          [\ZZhat{n,m}{+,\sigma}]_j(x)\,\ZZhat{n,m}{+,\sigma}(y)+ 
          [\ZZhat{n,m}{-,\sigma}]_j(x)\,\ZZhat{n,m}{-,\sigma}(y)\big)
\end{align} for $j=1,2$, where $(-)^\sigi=1$, $(-)^\sige=-1$. Then the Bergman operator $\BB_\N^\sigma$ for $\N^\sigma$ can be expressed as
\begin{align} \label{eq:bergmanopcontrag}
 \BB_\N^\sigma[f](x) =
  \left( \Sc \int_{\Omega}\,b_{\N,1}(x,y)f(y) \,dV_y\right)e_1 +
  \left( \Sc \int_{\Omega}\,b_{\N,2}(x,y)f(y) \,dV_y\right)e_2
\end{align}
for $x\in\B^\sigma$, $f\in L^2(\B^\sigma,\A_2)$. 

By construction, $b_{\N,j}^\sigma(x,y)$ is a contragenic function
of $x$ in $\B^\sigma$. To see that $\BB_\N^\sigma[f]$ is contragenic,
express \eqref{eq:Bfcalc} in the form
\begin{align*}
  \sum_k\int_\Omega(\psi_{1,k}e_1+\psi_{2,k}e_2) F_k(y) \, dV_y,
\end{align*}
where for our purposes the real function $F_k(y)$ is irrelevant,
and $\psi_{1,k}e_1+\psi_{2,k}e_2$ is contragenic. Now if
$g=g_1e_1+g_2e_2\in\Vec\M^\sigma$, then the inner product of
$\BB_\N^\sigma[f]$ with $g$ is
\begin{align*}
 \sum_k\int_\Omega\int_\Omega (g_1(x)\psi_1(x)+g_2(x)\psi_2(x))F_k(y)\, dV_y\, dV_x  
\end{align*}
which is seen to vanish after an application of Fubini's theorem
and the assumed contragenicity.

\begin{proposition}
  $\BB_\N^\sigma$ projects the Hilbert space
   $L^2(\B^\sigma,\R e_1+\R e_2)$ orthogonally onto the subspace
   $\N^\sigma$. Thus
  \begin{align*}
    \BB_\N^\sigma[f]=f
  \end{align*}
  if and only if $f\in\N^\sigma$. If $f$ is harmonic, then $\BB_\N^\sigma[f]=0$ if and only if
  $f\in\Vec\M^\sigma$.
\end{proposition}


\section{Examples\label{sec:examples}}

In this section, we present numerical  examples of the application of the Bergman operator $\BB_\M^\sigma$, calculated with \textit{Mathematica}.
In \cite[Eq.\ (10)]{FCM} the following monogenic version of the exponential function is given:
\begin{align}
  \E(x_0,x_1,x_2)&= e^{x_0}\bigg(
    \cos\big(\frac{x_1}{\sqrt{2}}\big)\cos\big(\frac{x_2}{\sqrt{2}}\big)\nonumber\\
   &\quad +\frac{1}{\sqrt{2}}\bigg(
      \sin\big(\frac{x_1}{\sqrt{2}}\big)\cos\big(\frac{x_2}{\sqrt{2}}\big)e_1
     +\cos\big(\frac{x_1}{\sqrt{2}}\big)\sin\big(\frac{x_2}{\sqrt{2}}\big)e_2
   \bigg)\bigg).  \label{eq:expmalonek}
\end{align}

It is immediate that the function
\begin{align*}
 \E^*(x_0,x_1,x_2) =  \E(-x_0,x_1,x_2)
\end{align*}
is also monogenic, and that  $\Vec\E$, $\Vec\E^*$ are in $L^2(\B^i)$,
$L^2(\B^e)$ respectively (see Figure \ref{fig:vecexp}).

\begin{figure}[!hb]
  \centering
  \pic{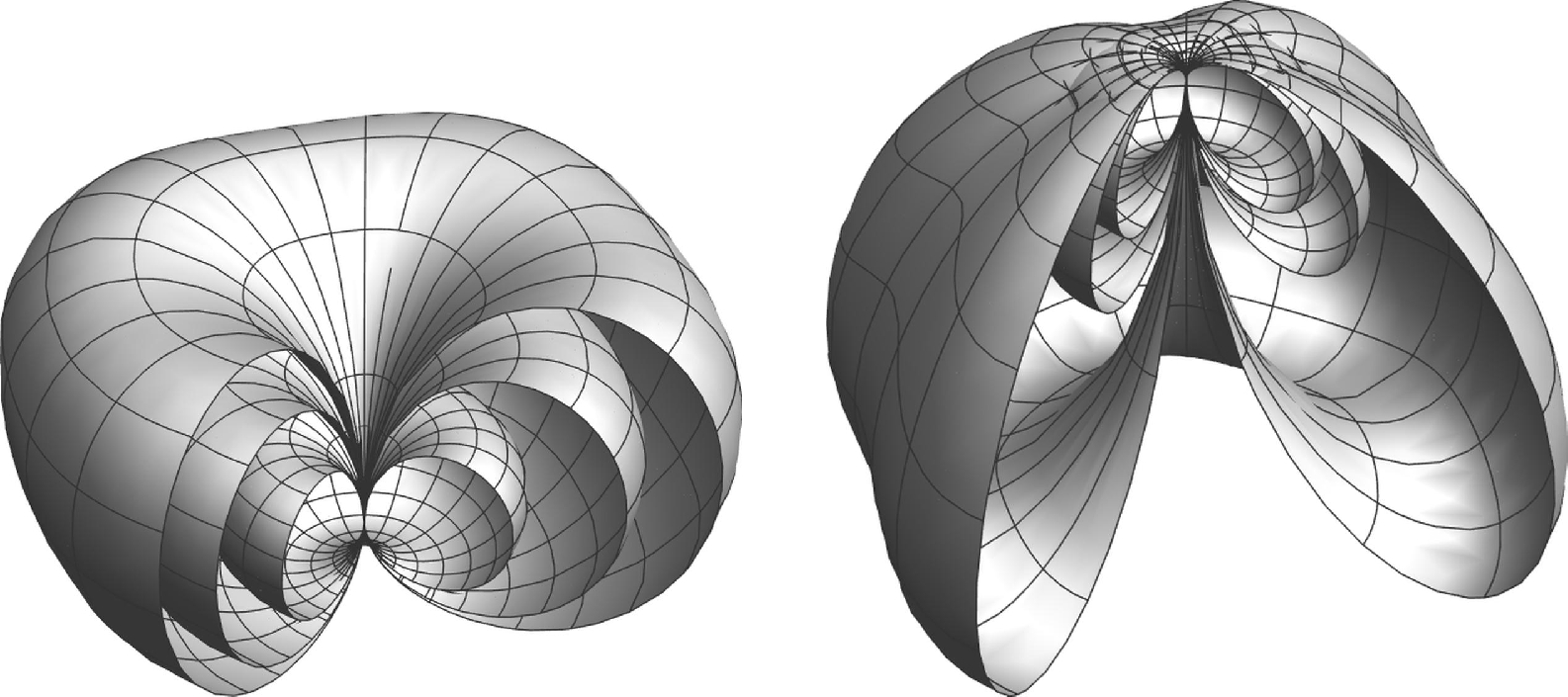}{-8.5,-6.8}{4.7}{scale=.2}
  \caption{Left: level surfaces of $\vert\Vec\E \vert$ for fixed values $\rho=0.5$, $0.8$, and $1.0$. Right: level surfaces of $\vert\Vec\E^*\vert$ for fixed values $\rho=1.0$, $1.25$, and $2.0$. }
  \label{fig:vecexp}
\end{figure}

\begin{table}[b!]
\[ \begin{array}{|c|c|c|c|c|c|c|c|l|}
\hline
\rho & N=5 & N=10 & N=15 & N=20\\
\hline
0.2 & 1.12\times10^{-5}&8.39\times10^{-6}& 8.39\times10^{-6}&8.39\times10^{-6}\\  
0.4& 1.28\times10^{-4} & 2.05\times10^{-5} & 2.02\times10^{-5} & 2.02\times10^{-5} \\  
0.6& 9.82\times10^{-4} & 2.51\times10^{-5} & 2.53\times10^{-5} & 2.48\times10^{-5} \\  
0.8& 4.86\times10^{-3} & 4.50\times10^{-4} & 9.93\times10^{-4} & 9.77\times10^{-4}\\ \hline 
\end{array} \]
\caption{Bergman operator for $\Vec\M^\sigi$: relative errors comparing $\Vec\E$ and truncated version of
  $\BB_\M^\sigi\Vec\E$, along spheres of radii $\rho$.}
  \label{tab:monogenicexp1}
\end{table}

In Tables \ref{tab:monogenicexp1} and \ref{tab:contragenicexp1} we
show the results of applying $\BB_\M^\sigi$ and $\BB_M^\sige$ to
$\Vec\E$, and then comparing the images to $\Vec\E$
and 0 respectively.  The generating formulas \eqref{eq:bergmanopmonog} and
\eqref{eq:bergmanopmonog} are truncated to the ranges $0\le n\le N$.
On each sphere of radius $\rho$, a selection of points approximately
uniformly distributed in $[0,\pi]\times[0,2\pi]$ were used to calculate
the maximum deviations. The reproducing property and the annihilating
property of the the two kernels are confirmed within roundoff and truncation error.

The closeness of the approximation can also be seen in Figures
\ref{fig:bergmanapproxint} and \ref{fig:bergmanapproxext}, which show
the images of the unit sphere $\partial\B^\sigi=\partial\B^\sige$
under the integral transform given by truncations of the series
\eqref{eq:bergmankercontrag} to a few terms.

\begin{table}[t!]
\[ \begin{array}{|c|c|c|c|c|c|c|c|c|c|l|}
\hline
\rho & N=15 & N=20 & N+25 & N=30\\
\hline
0.2 & 4.94\times10^{-8}&4.94\times10^{-8}&4.94\times10^{-8}&4.94\times10^{-8}\\  
0.4& 1.94\times10^{-7} & 1.94\times10^{-7}& 1.94\times10^{-7} &1.94\times10^{-7} \\  
0.6&  8.99\times10^{-7} & 8.96\times10^{-7} & 8.96\times10^{-7} &8.96\times10^{-7}  \\  
0.8& 4.83\times10^{-6} & 5.08\times10^{-6} & 4.99\times10^{-6} & 5.05\times10^{-6} \\\hline
\end{array} \]
\caption{ $\Vec\N^\sigi$: maximum values of the quotients of approximations of the $e_1$-component,
  $\vert[\BB_\N^\sigi \E]_1\vert/\vert[\E]_1\vert$.}
  \label{tab:contragenicexp1}
\end{table}

\begin{figure}[!ht]
  \centering
  \pic{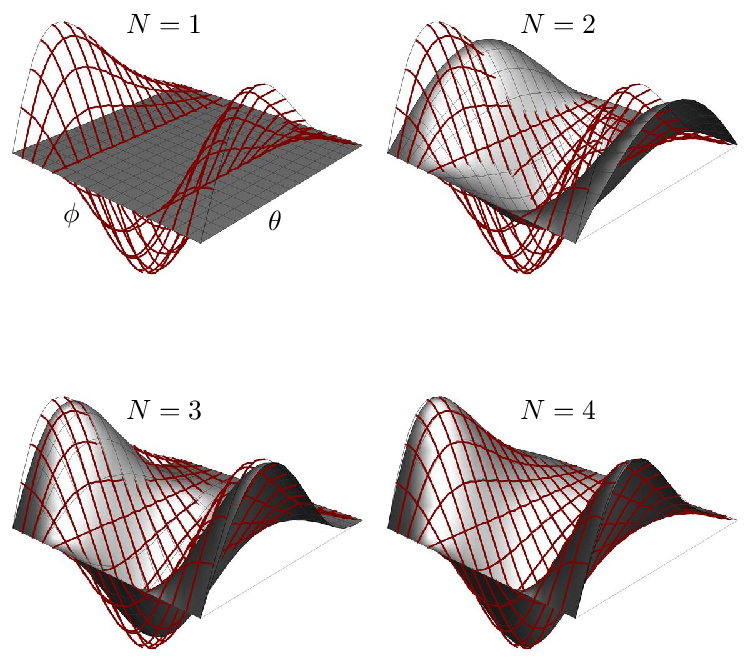}{-6.,-9.6}{6.5}{scale=1.3}
  \caption{Approximations using terms up to degree $N$ in
    $\BB_\M^\sigi(\Vec\E)$, graphed for $\rho=1$ and $(\theta,\phi)\in[0,\pi]\times[0,2\pi]$, $1\le N\le4$, and compared with the true value
  of $\Vec\E$.  For comparison, grid lines show exact values of  $\BB_\M^\sigi(\Vec\E)$.}
  \label{fig:bergmanapproxint}
\end{figure}

\begin{figure}[!ht]
  \centering
  \pic{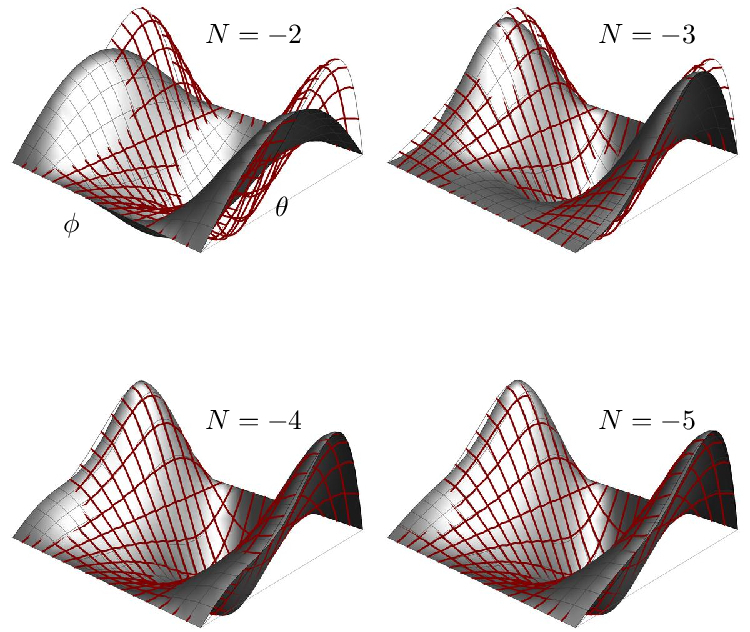}{-6,-9.6}{6.5}{scale=1.3}
  \caption{Analogous approximations for  $\BB_\M^\sige(\Vec\E^*)$.}
  \label{fig:bergmanapproxext}
\end{figure}


\section{Conclusions}

In the process of filling in the gaps of the known results on
square-integrable monogenic and contragenic functions defined on the
interior or the exterior of the unit ball, we have been able to
highlight the similarities and differences of the two cases.  In doing
so, we  uncovered an interesting duality relation among the basic
orthogonal functions (Proposition \ref{prop:duality}). This duality
relation permits expressing each of the four Bergman kernels (for
$\Vec \mathcal{M}$ and $\N$ corresponding to the interior and exterior of the sphere,
respectively) in two different ways.  All of these kernels serve to
detect contragenic functions, either by annihilating them or by
reproducing them as the case may be.

The exterior of the ball was found to admit contragenic functions which are not purely vectorial. This discovery leads one to pose the question of classifying domains in three-space according to whether they do or do not admit such contragenic functions.

The original kernel of S.\ Bergman, as described in
\cite{Bergman1950}, is defined for any planar domain via a series of
orthogonal holomorphic functions analogous to
\eqref{eq:bergmanopmonog}, but for the case of the unit disk in the
plane it takes the particularly elegant form
$1/(\pi(1-z\overline{w})^2)$. This leads one to wonder whether the
series expansions \eqref{eq:bergmanopmonog} and
\eqref{eq:bergmankercontrag} can admit similar expressions in closed
form.


\section{Acknowledgements}

\def\spc{\hspace{.7ex plus .2ex minus .5ex}}
The\spc second\spc author\spc acknowledges\spc financial\spc support\spc by\spc the\spc Asociaci\'on\spc Mexicana\spc de\spc Cultura,\spc A.\spc C.


\newcommand{\authorlist}[1]{#1,}
\newcommand{\booktitle}[1]{\textit{#1}.}
\newcommand{\articletitle}[1]{#1}
\newcommand{\journaltitle}[1]{\textit{#1}}
\newcommand{\volnum}[1]{\textbf{#1}}



\begin{thebibliography}{00}

\bibitem{Alvarez2013} \authorlist{C.\ \'Alvarez-Pe\~{n}a}
  \booktitle{Contragenic functions and Appell bases for monogenic
functions of three variables} doctoral dissertation, Cinvestav-I.P.N., Mexico City (2013).
  
\bibitem{Alvarez2014} \authorlist{C.\ \'Alvarez-Pe\~{n}a, R.\ M.\
Porter} \articletitle{Contragenic functions of three variables,}
\journaltitle{Complex Anal.\ Oper.\ Theory} \volnum{7}:1,
409--427 (2014).

\bibitem{AN1950} \authorlist{N.\ Aronszajn} \articletitle{Theory of
Reproducing Kernels,} \journaltitle{Trans. Amer. Math. Soc.}
\volnum{68}:3 (1950) 337--404. doi: 10.1090/S0002-9947-1950-0051437-7.
   
\bibitem{Axler} \authorlist{S.\ Axler, P.\ Bourdon, W.\ Ramey} \booktitle{Harmonic function theory} 
Springer-Verlag New York, Inc., 2nd edition (2001).

\bibitem{Bergman1950} \authorlist{S.\ Bergman} \booktitle{The kernel
function and conformal mapping} Mathematical Surveys \volnum{V},
American Mathematical Society, New York (1950).

\bibitem{BergmanSchiffer1953} \authorlist{S.\ Bergman, M.\ Schiffer}
\booktitle{Kernel Functions and elliptic differential equations in
mathematical physics} Academic Press Inc., New York (1953).
  
\bibitem{Cacao} \authorlist{I.\ Ca\c{c}\~{a}o} \booktitle{Constructive
approximation by monogenic polynomials} doctoral dissertation,
University of Aveiro (2004).

\bibitem{Cacao2010} \authorlist{I.\ Ca{\c c}{\~a}o}
\articletitle{Complete orthonormal sets of polynomial solutions of the
Riesz and Moisil-Teodorescu systems in $\mathbb{R}^3$,} \journaltitle{Numer.\
Algorithms} \volnum{55}:2-3, 191--203 (2010).

\bibitem{CGB2006} \authorlist{I.\ Ca{\c c}{\~a}o, K.\ Gürlebeck, S.\ Bock} \articletitle{On derivatives of spherical monogenics,}
  \journaltitle{Complex Var.\ Elliptic Equ.} \volnum{51}:8--11,
847--869 (2006).

\bibitem{DMP2021} \authorlist{M.\ Drewnik, T.\ Miller, Z.\
Pasternak-Winiarski} \articletitle{Reproducing kernel Hilbert space
associated with a unitary representation of a groupoid,}
\journaltitle{Complex Anal. Oper. Theory} \volnum{15}:91 (2021). doi:
/10.1007/s11785-021-01137-z
  
\bibitem{FCM} \authorlist{M.\ I.\ Falc\~{a}o, J.\ F.\ Cruz, H.\ R.\
Malonek} \articletitle{Remarks on the generation of monogenic
functions} 17th International Conference on the Application of
Computer Science and Mathematics in Architecture and Civil
Engineering, Weimar, Germany (2006).

\bibitem{FY2013} \authorlist{G.\ E.\ Fasshauer, Q.\ Ye}
\articletitle{Reproducing kernels of Sobolev spaces via a Green kernel
approach with differential operators and boundary operators,}
\journaltitle{Adv.\ Comput.\ Math.} \volnum{38}, 891--921 (2013).
doi: /10.1007/s10444-011-9264-6
  
\bibitem{GMP} \authorlist{R.\ Garc\'ia A., J.\ Morais, R.\ M.\ Porter}
\articletitle{Contragenic functions on spheroidal domains,}
\journaltitle{Math.\ Meth.\ in the Appl.\ Sci.} \volnum{41}:7,
2575--2589 (2018).

\bibitem{GMP2020} \authorlist{R.\ Garc\'ia A., J.\ Morais, R.\ M.\ Porter}
\articletitle{Relations among spheroidal and spherical harmonics,}
\journaltitle{Appl.\ Math.\ Comput.} \volnum{384}, 125147 (2020).

\bibitem{GLS2010} \authorlist{J.\ O.\ González-Cervantes, M.\ E.\
Luna-Elizarrarás, M.\ Shapiro} \articletitle{On the Bergman theory for
solenoidal and irrotational vector fields, I: General theory,} In: R.\
Duduchava, I.\ Gohberg, S.\ M.\ Grudsky, V.\ Rabinovich. (eds) Recent
Trends in Toeplitz and Pseudodifferential Operators. Operator Theory:
Advances and Applications, Vol.\ 210. Springer, Basel (2010).

\bibitem{Gurlebeck2} \authorlist{K.\ G\"{u}rlebeck, K.\ Habetha, W.\
Spr\"{o}ssig} \booktitle{Holomorphic functions in the plane and
$n$-dimensional space} Birkh\"auser Verlag, Basel-Boston-Berlin (2008).

\bibitem{GuerlebeckHabethaSproessig2016} \authorlist{K.\ G\"urlebeck,
K.\ Habetha, W.\ Spr\"o{\ss}ig} \booktitle{Application of holomorphic
functions in two and higher dimensions} Birkh\"auser Verlag,
Basel-Boston-Berlin (2016).

\bibitem{Hobson} \authorlist{E.\ Hobson} \booktitle{The theory of
spherical and ellipsoidal harmonics} Cambridge (1931).

  
\bibitem{HZ2009} \authorlist{L.\ Han, X.\ Zhang}
\articletitle{Construction and calculation of reproducing kernel
determined by various linear differential operators}
\journaltitle{Appl. Math. Comput.} \volnum{215}:2, 759--766 {2009}.
doi: 10.1016/j.amc.2009.05.063

\bibitem{KT2022} \authorlist{M.\ T.\ Karaev, R.\ Tapdigoglu}
\articletitle{On some problems for reproducing kernel Hilbert space
operators via the Berezin transform} \journaltitle{Mediterr.\ J.\
Math.} \volnum{19}:13 (2022). doi: /10.1007/s00009-021-01926-y

\bibitem{MW2014} \authorlist{M.\ Mitkovski, B.\ D.\ Wick}
\articletitle{A reproducing kernel thesis for operators on
Bergman-type function spaces} \journaltitle{J.\ Funct.\ Anal.}
\volnum{267}:7 2028--2055, (2014).  doi: 10.1016/j.jfa.2014.07.020

\bibitem{MSKS2020} \authorlist{M.\ Mollenhauer, I.\ Schuster, S.\
Klus, C.\ Schütte} \articletitle{Singular value decomposition of
pperators on reproducing kernel Hilbert spaces} In: O.\ Junge, et.\
al.\ (eds) \booktitle{Advances in dynamics, optimization and
computation, SON 2020} \journaltitle{Studies in Systems, Decision and
Control} \volnum{304}, Springer, Cham. (2020). doi:
10.1007/978-3-030-51264-4\_5
  
\bibitem{MoraisGur} \authorlist{J.\ Morais, K.\ G\"{u}rlebeck}
\articletitle{Real-part estimates for solutions of the Riesz system in
$\R^3$,} \journaltitle{Complex Var.\ Elliptic Equ.} \volnum{57}:5,
505--522 (2012).

\bibitem{MoraisNguyenKou2016} \authorlist{J.\ Morais, M.\ H.\ Nguyen, K.\ I.\ Kou} \articletitle{On 
3D orthogonal prolate spheroidal monogenics,} \journaltitle{Math. Methods Appl. Sci.} 
\volnum{39}:4, 635--648 (2016).

\bibitem{MoraisHabilitation2021} \authorlist{J.\ Morais} \booktitle{A
quaternionic version theory related to spheroidal functions}
Habilitation thesis, TU Bergakademie Freiberg, 2021.

\bibitem{Muller} \authorlist{C.\ M\"{u}ller} \booktitle{Spherical
harmonics} Lectures notes in mathematics, \volnum{17} Berlin:
Springer-Verlag (1966).

\bibitem{Paulsen2016} \authorlist{V.\ Paulsen} \articletitle{An
introduction to the theory of reproducing kernel Hilbert spaces}
Cambridge University Press, Cambridge (2016) doi:
10.1017/CBO9781316219232
  
\bibitem{Sansone} \authorlist{G.\ Sansone} \booktitle{Orthogonal
functions} Pure and Applied Mathematics, \volnum{IX}, Interscience
Publishers, New York (1959).

\bibitem{Sudbery} \authorlist{A.\ Sudbery} \articletitle{Quaternionic analysis,} \journaltitle{ Math Proc Cambridge Phil Soc.} \volnum{85}, 199--225 (1979).

\end{thebibliography}
\end{document}